\documentclass{amsart}

\usepackage{amssymb,amsmath,graphicx,dsfont,color,url}

\newtheorem{theorem}{Theorem}[section]
\newtheorem{corollary}[theorem]{Corollary}
\newtheorem{lemma}[theorem]{Lemma}
\newtheorem{proposition}[theorem]{Proposition}
\newtheorem*{questions}{Questions}

\theoremstyle{definition}
\newtheorem{definition}[theorem]{Definition}

\theoremstyle{remark}
\newtheorem{remark}[theorem]{Remark}
\newtheorem*{Remark}{Remark}

\numberwithin{equation}{section}

\begin{document}

\hspace{1in}

\title{Fibrations, unique path lifting, and continuous monodromy}

\author{Hanspeter Fischer}

\address{Department of Mathematical Sciences, Ball State University, Muncie, IN 47306, USA}

\email{hfischer@bsu.edu}

\author{Jacob D. Garcia}

\address{Department of Mathematical Sciences, Ball State University, Muncie, IN 47306, USA}

\email{jgarc351@ucr.edu}

\thanks{}



\date{\today}

\commby{}

\begin{abstract}
Given a path-connected space $X$ and  $H\leq\pi_1(X,x_0)$,
 there is essentially only one construction of a map
 $p_H:(\widetilde{X}_H,\widetilde{x}_0)\rightarrow(X,x_0)$
 with connected and locally path-connected domain that can possibly have the following two properties: $(p_{H})_{\#}\pi_1(\widetilde{X}_H,\widetilde{x}_0)=H$ and $p_H$ has the unique lifting property. $\widetilde{X}_H$ consists of equivalence classes of paths starting at $x_0$, appropriately topologized, and $p_H$ is the endpoint projection. For $p_H$ to have these two properties, $T_1$ fibers are necessary and unique path lifting is sufficient. However, $p_H$ always admits the standard lifts of paths.

We show that $p_H$ has unique path lifting if it has continuous (standard) monodromies toward a $T_1$ fiber over $x_0$. Assuming, in addition, that $H$ is locally quasinormal (e.g., if $H$ is normal) we show that $X$ is homotopically path Hausdorff relative to $H$. We show that $p_H$ is a fibration if $X$ is locally path connected, $H$ is locally quasinormal, and all (standard) monodromies are continuous.

\end{abstract}

 \maketitle

\section{Introduction, Notation, and Terminology}\label{intro}

Recent work on generalizations of covering space theory has generated interest in the relationships between individual properties that are otherwise shared by classical covering projections. One such generalization focuses on the existence and properties of maps from connected and locally path-connected spaces to path-connected spaces that satisfy the unique lifting property (see below).
They can serve as suitable substitutes when classical covering projections are not available, as they retain much of the classical utility of the original theory (such as automorphisms being tied to fundamental groups) while sacrificing less critical features \cite{B, BFi, FZ2007}.

In this article, we explore the connections between the concepts of fibration, unique path lifting, and continuous monodromy in the context of the standard covering construction, when applied to a general path-connected space (which may or may not yield a covering projection, or even a generalized covering projection). To begin, we lay out some notation and terminology.

\vspace{5pt}

\noindent {\bf General Assumption:} Throughout, we let $X$ be a path-connected topological space with base point $x_0\in X$.
\vspace{5pt}

We refer to \cite{Spanier} for  definitions and basic properties of covering projections and fibrations.
Every covering projection $p:\widetilde{X}\rightarrow X$ is a (Hurewicz) fibration with unique path lifting. Indeed, much of the classical development of covering space theory is based on the fact that if $p:\widetilde{X}\rightarrow X$ is a fibration with unique path lifting, then we have the following:
\begin{itemize}
\item (Unique Lifting Property) For every $x\in X$ and $\widetilde{x}\in \widetilde{X}$ with $p(\widetilde{x})=x$  and for every map $f:(Y,y)\rightarrow (X,x)$ from a connected and locally path-connected space $Y$   with $f_\#\pi_1(Y,y)\leq p_\#\pi_1(\widetilde{X},\widetilde{x})$, there is a {\sl unique} map $g:Y\rightarrow \widetilde{X}$ with $g(y)=\widetilde{x}$ and $p\circ g=f$.
\item (Monodromy) For every path $\beta:[0,1]\rightarrow X$ from $\beta(0)=x$ to $\beta(1)=y$,\linebreak  there is a continuous {\em monodromy} $\phi_{\beta}: p^{-1}(x)\rightarrow p^{-1}(y)$, defined by \linebreak $\phi_\beta(\widetilde{x})=g(1)$, where $g:[0,1]\rightarrow \widetilde{X}$ is the unique path with $g(0)=\widetilde{x}$ and $p\circ g=\beta$; moreover, $\phi_\beta$  depends only on the homotopy class of $\beta$.
\end{itemize}

A map $p:\widetilde{X}\rightarrow X$ from a nonempty, connected, and locally path-connected space $\widetilde{X}$ is called a {\em generalized covering projection} if it satisfies the unique lifting property. Observe that a generalized covering projection $p:(\widetilde{X},\widetilde{x}_0)\rightarrow (X,x_0)$ is uniquely characterized, up to equivalence, by the conjugacy class of the monomorphic image $H=p_\#\pi_1(\widetilde{X},\widetilde{x}_0)$ in $G=\pi_1(X,x_0)$ and we have Aut$(\widetilde{X} \stackrel{p}{\rightarrow} X)\approx N_G(H)/H$.

If $X$ is locally path connected, then a generalized covering projection $p:\widetilde{X}\rightarrow X$  is necessarily an open map, but it might not be a local homeomorphism, might not be a fibration and, while the fibers must be T$_1$ (see Remark~\ref{HH}), there might be non-homeomorphic fibers. (See \cite{FZ2007}.)
For example, there is a generalized covering projection $p:\widetilde{\mathbb{H}}\rightarrow \mathbb{H}$ over the Hawaiian Earring $\mathbb{H}$ with simply connected $\widetilde{\mathbb{H}}$,
which is not a fibration, since not all monodromies are  continuous; in fact, it has exactly one exceptional non-discrete fiber.
There also exist generalized covering projections over $\mathbb{H}$ that are local homeomorphisms (with necessarily discrete fibers) but have  other unexpected properties, such as $H$  not containing any nontrivial normal subgroup of $G$. (See \cite{FZ2013}.)

As is shown in \cite{B}, given a subgroup $H\leq \pi_1(X,x_0)$, there is (up to equivalence) only one construction of a map $p_H:(\widetilde{X}_H,\widetilde{x}_0)\rightarrow (X,x_0)$ that can possibly yield a generalized covering projection with  $(p_{H})_{\#}\pi_1(\widetilde{X}_H,\widetilde{x}_0)=H$, namely the standard one:

 On the set of all based paths $\alpha:([0,1],0)\rightarrow (X,x_0)$, consider the equivalence relation $\alpha\sim \beta$ if and only if $\alpha(1)=\beta(1)$ and $[\alpha\cdot \beta^-]\in H$, where $\beta^-(t)=\beta(1-t)$. Denote the equivalence class of $\alpha$ by $\left<\alpha\right>$ and let $\widetilde{X}_H$ denote the set of all equivalence classes. We endow $\widetilde{X}_H$ with the topology generated by  basis elements of the form $\left<\alpha,U\right>=\{\left<\alpha\cdot \gamma \right> \mid \gamma:([0,1],0)\rightarrow (U,\alpha(1))\}$, where $U$ is an open subset of $X$ and $\left<\alpha\right>\in \widetilde{X}_H$ with $\alpha(1)\in U$. Here, $\alpha\cdot \gamma$ denotes the usual concatenation of the paths $\alpha$ and $\gamma$. One observes that $\widetilde{X}_H$ is connected and locally path connected and that the endpoint projection $p_H:\widetilde{X}_H\rightarrow X$, defined by $p_H(\left<\alpha\right>)=\alpha(1)$, is continuous.

If $p_H:\widetilde{X}_H\rightarrow X$ has unique path lifting, then it is a generalized covering projection and $(p_{H})_{\#}\pi_1(\widetilde{X}_H,\widetilde{x}_0)=H$, where $\widetilde{x}_0$ is the equivalence class represented by the constant path.

Even if $p_H:\widetilde{X}_H\rightarrow X$ does not have unique path lifting, the standard (continuous) lifts always exist when permitted by the $\pi_1$-functor. For example, the standard lift of a path $\beta:[0,1]\rightarrow X$ at $\left<\alpha\right>\in p_H^{-1}(\beta(0))$ is given by $g(t)=\left<\alpha\cdot \beta_t\right>$, where $\beta_t(s)=\beta(ts)$.

We may use the standard path lift to define the {\em standard monodromy}: For a path $\beta:[0,1]\rightarrow X$ from $\beta(0)=x$ to $\beta(1)=y$, we define $\phi_\beta:p_H^{-1}(x)\rightarrow p_H^{-1}(y)$ by $\phi_\beta(\left<\alpha\right>)=\left<\alpha\cdot \beta\right>$. While $\phi_\beta$ is clearly a bijection (with inverse $\phi_\beta^{-1}=\phi_{\beta^-}$), it might or might not be a continuous function.

  This raises the following questions.

 \begin{questions} Suppose  $p_H:\widetilde{X}_H\rightarrow X$ has T$_1$ fibers and  its standard monodromies are continuous. Does $p_H$ have unique path lifting? Is $p_H$ a fibration?
\end{questions}

In \S\ref{UPL}, we show that $p_H:\widetilde{X}_H\rightarrow X$ has unique path lifting if it has T$_1$ fibers and the standard monodromy $\phi_\beta$ is continuous for all $\beta$ with $\beta(1)=x_0$.
In \S\ref{fib}, we show that  $p_H:\widetilde{X}_H\rightarrow X$  is a fibration if $X$ is locally path connected, $H$ is locally quasinormal, and all monodromies are continuous. (See Definition~\ref{LQN} for the concept of local quasinormality. For now, we mention two examples: if $H$ is normal or if $p_H:\widetilde{X}_H\rightarrow X$ has discrete fibers, then $H$ is locally quasinormal.)

We also show, in \S\ref{HpH}, that $X$ is homotopically path Hausdorff relative to $H$ (see Definition~\ref{HpHDef}) if $H$ is locally quasinormal at the constant path (e.g., if $H$ is locally quasinormal), $p_H:\widetilde{X}_H\rightarrow X$ has T$_1$ fibers and the standard monodromy $\phi_\beta$ is continuous for all $\beta$ with $\beta(1)=x_0$.

The concept of $p_H:\widetilde{X}_H\rightarrow X$ having continuous monodromies can equivalently be described in terms of small loop transfer (SLT) in $X$ relative to $H$, as developed in \cite{BDLM} and \cite{PMTAR} (see Definition~\ref{SLT-Def} and Lemmas~\ref{SLT-Mon}--\ref{SLT-N}). Using this alternate concept, it is shown in \cite[Theorem~2.5]{PMTAR} that $X$ is homotopically path Hausdorff relative to~$H$ if $H\trianglelefteq \pi_1(X,x_0)$ is a normal subgroup, $X$ is homotopically Hausdorff relative to~$H$ (see Remark~\ref{HH}), and $X$ is an $H$-SLT space at $x_0$. In \S\ref{SLT}, we examine how our results generalize and extend Theorem~2.5 of \cite{PMTAR}.

\section{Continuous monodromy and unique path lifting}\label{UPL}
\begin{remark} \label{HH} If $p_H:\widetilde{X}_H\rightarrow X$ has unique path lifting, then for every $x\in X$, the fiber $p_H^{-1}(x)$  is $T_1$. The short proof of this fact can be found in \cite[Proposition~6.4]{FZ2007}, where $X$ is called {\em homotopically Hausdorff relative to $H$} if every fiber of  $p_H$ is $T_1$.
\end{remark}

In this section, we prove the following implication.

\begin{theorem}\label{thm1} Let $H\leq \pi_1(X,x_0)$. Suppose that $p_H^{-1}(x_0)$ is $T_1$ and that for every $x\in X$ and for every path $\beta:[0,1]\rightarrow X$ from $\beta(0)=x$ to $\beta(1)=x_0$, the monodromy $\phi_\beta:p_H^{-1}(x)\rightarrow p_H^{-1}(x_0)$ is continuous. Then $p_H:\widetilde{X}_H\rightarrow X$ has unique path lifting.
\end{theorem}

We begin with a straightforward fact, whose proof we include for completeness:

\begin{lemma}\label{disconnected}
Let $H\leq \pi_1(X,x_0)$ and $x\in X$. The fiber $p_H^{-1}(x)$ is T$_1$ if and only if all of its quasicomponents are one-point sets.
\end{lemma}

\begin{proof}
  Observe that if $\left<\beta\right>\in \left<\alpha,U\right>$, then $\left<\beta,U\right>=\left<\alpha,U\right>$. Therefore, two basis elements of $\widetilde{X}_H$ of the form $\left<\alpha,U\right>$ and $\left<\beta,U\right>$ are either disjoint or equal.
Now assume that $p_H^{-1}(x)$ is T$_1$ and let $\left<\alpha\right>$ and $\left<\beta\right>$ be two distinct elements of $p_H^{-1}(x)$. Then there is an open neighborhood $U$ of $x$ in $X$ such that $\left<\beta\right>\not\in \left<\alpha,U\right>\cap p_H^{-1}(x)$.
 Put $W_1=\left<\alpha,U\right>\cap p_H^{-1}(x)$ and $W_2=\bigcup \{ \left<\gamma,U\right>\cap p_H^{-1}(x)\mid \left<\gamma\right>\in p_H^{-1}(x), \left<\gamma\right>\not\in \left<\alpha,U\right> \}$. Then $W_1$ and $W_2$ are open subsets of $p_H^{-1}(x)$ with $\left<\alpha\right>\in W_1$, $\left<\beta\right>\in W_2$, $W_1\cap W_2=\emptyset$ and $W_1\cup W_2=p_H^{-1}(x)$. The converse is trivial.
\end{proof}

\begin{proposition}\label{fiberpath}
Let $H\leq \pi_1(X,x_0)$ and let $f,g:[0,1]\rightarrow \widetilde{X}_H$ be two lifts of a path $\beta:[0,1]\rightarrow X$ with $p_H\circ f=\beta=p_H\circ g$ and $f(0)=g(0)$.

For each $t\in [0,1]$, choose a path $\alpha_t:([0,1],0)\rightarrow (X,x_0)$ so that $f(t)=\left<\alpha_t\right>$ and assume that $g$ is the standard lift with $g(t)=\left<\alpha_0\cdot\beta_t\right>$, where $\beta_t(s)=\beta(ts)$.

If the monodromy $\phi_{\beta_t^-}:p_H^{-1}(\beta(t))\rightarrow p_H^{-1}(\beta(0))$ is continuous for every $t\in [0,1]$, then the function $h:[0,1]\rightarrow p_H^{-1}(\beta(0))$, given by $h(t)=\left<\alpha_t\cdot \beta_t^-\right>$, is continuous.
\end{proposition}

\begin{proof} Let $t\in [0,1]$. Put $\gamma=\alpha_t\cdot \beta_t^-$. Then $h(t)=\left<\gamma\right>$. Let $U$ be an open neighborhood of $\gamma(1)=\beta(0)$ in $X$. By continuity of $\phi_{\beta_t^-}$, there is an open neighborhood $V$ of $\alpha_t(1)=  \beta(t)$ in $X$ such that $\phi_{\beta_t^-}(\left<\alpha_t,V \right>\cap p_H^{-1}(\beta(t))\subseteq \left<\gamma,U\right>$. By continuity of $\beta$ and $f$, there is an interval $I$, open in $[0,1]$, with $t\in I$ such that $\beta(I)\subseteq V$ and $f(I)\subseteq \left<\alpha_t,V\right>$. Let $s\in I$. Then $f(s)=\left<\alpha_s\right>=\left<\alpha_t\cdot \delta_1\right>$ for some $\delta_1:[0,1]\rightarrow V$. Moreover, $g(s)=\left<\alpha_0\cdot \beta_s\right>=\left<\alpha_0\cdot \beta_t\cdot \delta_2\right>$, where $[\beta_s]=[\beta_t\cdot \delta_2]$ with $\delta_2:[0,1]\rightarrow \beta(I)\subseteq V$. Hence, $h(s)=\left<\alpha_s\cdot \beta_s^-\right>=\left<\alpha_t\cdot \delta_1\cdot \delta_2^-\cdot \beta_t^-\right>=\phi_{\beta_t^-}(\left<\alpha_t\cdot \delta_1\cdot \delta_2^-\right>)\in \left<\gamma,U\right>$.
\end{proof}

\begin{corollary}\label{betaUPL}
Let $H\leq \pi_1(X,x_0)$ and let $\beta:[0,1]\rightarrow X$ be a path. Suppose that \begin{itemize} \item[(i)] $\phi_{\beta_t^-}: p_H^{-1}(\beta(t))\rightarrow p_H^{-1}(\beta(0))$ is continuous for all $t\in [0,1]$; and \item[(ii)]  $p_H^{-1}(\beta(0))$ is T$_1$.\end{itemize}
Then $p_H^{-1}(\beta(t))$ is T$_1$ for all $t\in [0,1]$ and $\beta$ has unique lifts: if $f,g:[0,1]\rightarrow \widetilde{X}_H$ are two paths with $p_H\circ f=\beta=p_H\circ g$ and $f(0)=g(0)$, then $f=g$.
\end{corollary}

\begin{proof}
The first assertion follows from the fact that the monodromy $\phi_{\beta_t^-}$ is a continuous bijection. As for the unique lifting of $\beta$, we may assume that $f,g,\alpha_t$ and $\beta_t$ are as in Proposition~\ref{fiberpath}, so that $h:[0,1]\rightarrow p_H^{-1}(\beta(0))$, given by $h(t)=\left<\alpha_t\cdot \beta_t^-\right>$, is continuous. Since $p_H^{-1}(\beta(0))$ is T$_1$, Lemma~\ref{disconnected} implies that $h$ is constant. Hence, for all $t\in [0,1]$, $\left<\alpha_0\right>=\left<\alpha_0\cdot \beta_0^-\right>=\left<\alpha_t\cdot \beta_t^-\right>$ and $f(t)=\left<\alpha_t\right>=\left<\alpha_0\cdot \beta_t\right>=g(t)$.
\end{proof}

\begin{proof}[Proof of Theorem~\ref{thm1}] In order to prove unique path lifting, we may restrict our attention to paths $\beta:[0,1]\rightarrow X$ with $\beta(0)=x_0$. Apply Corollary~\ref{betaUPL}.
\end{proof}

\section{Locally quasinormal subgroups and fibrations}\label{fib}

For a path $\alpha:([0,1],0)\rightarrow (X,x_0)$ and an open neighborhood $U$ of $\alpha(1)$ in $X$, we define the subgroup \[\pi(\alpha,U)=\{[\alpha\cdot \delta\cdot \alpha^-]\mid \delta:([0,1],\{0,1\})\rightarrow (U,\alpha(1))\}\leq \pi_1(X,x_0).\]

\begin{definition} \label{LQN} We call a subgroup $H\leq \pi_1(X,x_0)$ {\em locally quasinormal} if for every $x\in X$, for every path $\alpha:[0,1]\rightarrow X$ from $\alpha(0)=x_0$ to $\alpha(1)=x$, and for every open neighborhood $U$ of $x$ in $X$, there is an open neighborhood $V$ of $x$ in $X$ with $x\in V\subseteq U$ such that $H\pi(\alpha,V)=\pi(\alpha,V)H$.
\end{definition}

\begin{Remark}
Recall from elementary group theory that for subgroups $H,K \leq G$, one has $HK=KH$ $\Leftrightarrow$ $HK\subseteq KH$ $\Leftrightarrow$ $HK\supseteq KH$ $\Leftrightarrow$ $HK\leq G$ and that one calls $H$ {\em quasinormal} if $HK=KH$ for every $K\leq G$.
\end{Remark}

A subgroup $H \leq \pi_1(X,x_0)$ is normal if and only if $\phi_\beta:p_H^{-1}(\alpha(1))\rightarrow p_H^{-1}(\alpha(1))$ is the identity function whenever $\phi_\beta(\left<\alpha\right>)=\left<\alpha\right>$. Local quasinormality, on the other hand, can be understood as a uniform local stability condition for the set of monodromies that fix a given point:

\begin{lemma}\label{stable}
Let $H\leq \pi_1(X,x_0)$, let $\alpha:[0,1]\rightarrow X$ be a path from $\alpha(0)=x_0$ to $\alpha(1)=x$, and let $U$ be an open neighborhood of $x$ in $X$. Then
 $H\pi(\alpha,U)=\pi(\alpha,U)H$ if and only if
 for every monodromy $\phi_\beta:p_H^{-1}(x)\rightarrow p_H^{-1}(x)$ with $\phi_\beta(\left<\alpha\right>)=\left<\alpha\right>$, we have  $\phi_\beta(\left<\alpha,U\right>\cap p_H^{-1}(x))=\left<\alpha,U\right>\cap p_H^{-1}(x)$.
\end{lemma}

We include the straightforward proof for completeness:

\begin{proof}

(a) Suppose $H\pi(\alpha,U)\supseteq \pi(\alpha,U)H$ and $\phi_\beta(\left<\alpha\right>)=\left<\alpha\right>$. It suffices to show that  $\phi_\beta(\left<\alpha,U\right>\cap p_H^{-1}(x))\subseteq \left<\alpha,U\right>\cap p_H^{-1}(x)$, since $\phi_{\beta^-}=\phi_\beta^{-1}$. Let $\left<\eta\right>\in \left<\alpha,U\right>\cap p_H^{-1}(x)$. Then $\left<\eta\right>=\left<\alpha\cdot \delta\right>$ for some loop $\delta:[0,1]\rightarrow U$ with $\delta(0)=\delta(1)=x$.
We have $[\alpha\cdot \delta \cdot \alpha^-][\alpha\cdot \beta \cdot \alpha^-]\in \pi(\alpha,U)H\subseteq H\pi(\alpha,U)$. Hence, $[\alpha\cdot \delta \cdot \beta \cdot \alpha^-]=[\gamma][\alpha\cdot \widehat{\delta} \cdot \alpha^-]$ for some $[\gamma]\in H$ and some loop $\widehat{\delta}:[0,1]\rightarrow U$ with $\widehat{\delta}(0)=\widehat{\delta}(1)=x$.
 Therefore, $[\alpha\cdot \delta \cdot \beta \cdot  \widehat{\delta}^- \cdot \alpha^-]\in H$ and $\phi_\beta(\left<\eta\right>)=\left<\alpha\cdot \delta\cdot \beta\right>=\left<\alpha\cdot \widehat{\delta}\right>\in \left<\alpha,U\right>\cap p_H^{-1}(x)$.

(b) Suppose $\phi_\beta(\left<\alpha,U\right>\cap p_H^{-1}(x))\subseteq \left<\alpha,U\right>\cap p_H^{-1}(x)$ whenever $\phi_\beta(\left<\alpha\right>)=\left<\alpha\right>$. We show that $H\pi(\alpha,U)\supseteq \pi(\alpha,U)H$. Let $[\tau]\in \pi(\alpha,U)H$. Then $[\tau]=[\alpha\cdot \delta \cdot \alpha^-\cdot \gamma]$ for some loop $\delta:[0,1]\rightarrow U$ with $\delta(0)=\delta(1)=x$ and some $[\gamma]\in H$. Put $\beta=\alpha^-\cdot \gamma\cdot \alpha$. Then $\phi_\beta(\left<\alpha\right>)=\left<\alpha\right>$ and  $\left<\alpha\cdot \delta\right>\in \left<\alpha,U\right>\cap p_H^{-1}(x)$, so that  $\left<\alpha\cdot \delta\cdot \beta\right>=\phi_\beta(\left<\alpha\cdot \delta\right>)\in \left<\alpha,U\right>\cap p_H^{-1}(x)$. Hence, $\left<\alpha\cdot \delta\cdot \beta\right> = \left<\alpha\cdot \widehat{\delta}\right>$ for some loop $\widehat{\delta}:[0,1]\rightarrow U$ with $\widehat{\delta}(0)=\widehat{\delta}(1)=x$, so that
 $[\tau]=[\alpha\cdot \delta\cdot \beta\cdot \alpha^-]=[\alpha\cdot \delta \cdot \beta \cdot \widehat{\delta}^-\cdot \alpha^-][\alpha\cdot \widehat{\delta}\cdot \alpha^-]\in H\pi(\alpha,U)$.
\end{proof}

We note two important instances of local quasinormality:

\begin{lemma}
If  $H \trianglelefteq \pi_1(X,x_0)$ is a normal subgroup or if $p_H:\widetilde{X}_H\rightarrow X$ has discrete fibers, then $H$ is a locally quasinormal subgroup of $\pi_1(X,x_0)$.
\end{lemma}

\begin{proof} This follows from Definition~\ref{LQN} and Lemma~\ref{stable}, respectively. \end{proof}

In this section, we prove the following implication.

\begin{theorem} \label{thm2} Let $X$ be locally path connected. Suppose that $H\leq \pi_1(X,x_0)$ is locally quasinormal and that for every $x,y\in X$ and for every path $\beta:[0,1]\rightarrow X$ from $\beta(0)=x$ to $\beta(1)=y$, the monodromy $\phi_\beta:p_H^{-1}(x)\rightarrow p_H^{-1}(y)$ is continuous.  Then $p_H:\widetilde{X}_H\rightarrow X$ is a fibration.
\end{theorem}

\begin{proof}
Let $Y$ be any topological space and let $\widetilde{f}:Y\rightarrow \widetilde{X}_H$ and $F:Y\times [0,1]\rightarrow X$ be continuous maps such that $p_H\circ \widetilde{f}(y)=F(y,0)$ for every $y\in Y$.

We define a function $\widetilde{F}:Y\times [0,1]\rightarrow \widetilde{X}_H$ as follows.
For each $y\in Y$, choose a path $\alpha_y:([0,1],0)\rightarrow (X,x_0)$ such that $\widetilde{f}(y)=\left<\alpha_y\right>$. For $y\in Y$ and $t_1, t_2\in [0,1]$ let $\beta_{y,t_1,t_2}:[0,1]\rightarrow X$ be the path defined by $\beta_{y,t_1,t_2}(s)=F(y,t_1+s(t_2-t_1))$. For $t\in [0,1]$, put $\beta_{y,t}=\beta_{y,0,t}$. Since $\alpha_y(1)=p_H\circ \widetilde{f}(y)=F(y,0)=\beta_{y,t}(0)$, we may define $\widetilde{F}(y,t)=\left<\alpha_y\cdot \beta_{y,t}\right>$.
Then $\widetilde{F}(y,0)=\left<\alpha_y\right>=\widetilde{f}(y)$ for every $y\in Y$ and $p_H\circ \widetilde{F}(y,t)=\beta_{y,t}(1)=F(y,t)$ for every $(y,t)\in Y\times[0,1]$. It remains to show that $\widetilde{F}$ is continuous.

Let $(y_0,t_0)\in Y\times [0,1]$. For simplicity, put $\alpha=\alpha_{y_0}$, $\beta_{t_1,t_2}=\beta_{y_0,t_1,t_2}$, $\beta_t=\beta_{y_0,t}$ and $\beta=\beta_{y_0,1}$. Then $\widetilde{F}(y_0,t_0)=\left<\alpha\cdot \beta_{t_0}\right>$.
Let $U$ be an open neighborhood of $\beta_{t_0}(1)=F(y_0,t_0)=\beta(t_0)$ in $X$. We wish to find open subsets $M\subseteq Y$ and $I\subseteq [0,1]$ with $(y_0,t_0)\in M\times I$ such that $\widetilde{F}(M\times I)\subseteq \left<\alpha \cdot \beta_{t_0} ,U\right>$. For this, since $H$ is locally quasinormal, we may assume that $H\pi(\alpha\cdot \beta_{t_0} ,U)=\pi(\alpha \cdot \beta_{t_0},U) H$.

For each $t\in [0,t_0]$, the monodromy $\phi_{\beta_{t,t_0}}:p_H^{-1}(\beta(t))\rightarrow p_H^{-1}(\beta(t_0))$ is continuous, so that there is an open neighborhood $U_t$ of $\beta(t)$ in $X$ such that \begin{equation} \label{star} \phi_{\beta_{t,t_0}}(\left<\alpha\cdot \beta_t,U_t\right>\cap p_H^{-1}(\beta(t)))\subseteq \left<\alpha\cdot \beta_{t_0},U\right>\cap p_H^{-1}(\beta(t_0)). \end{equation} We may assume that $U_{t_0}=U$. For each $t\in [0,t_0]$, choose an interval $W_t$, open in $[0,t_0]$, with  $t\in W_t\subseteq cl(W_t)\subseteq \beta^{-1}(U_t)$. Say, $cl(W_{t_0})=[w_0,t_0]$. Since $[0,w_0]$ is compact, there is a finite subdivision $0=s_0<s_1<\cdots <s_{n-1}=w_0$ such that for each $1\leq i \leq n-1$, there is a $t_i\in [0,w_0]$ with $[s_{i-1},s_i]\subseteq W_{t_i}$.
 For $1\leq i \leq n-1$, put $U_i=U_{t_i}$.
 By (\ref{star}) and Lemma~\ref{slide}, for  $1\leq i \leq n-1$, we have  \begin{equation}\label{starstar} \phi_{\beta_{s_i,t_0}}(\left<\alpha\cdot \beta_{s_i},U_i\right>\cap p_H^{-1}(\beta(s_i)))\subseteq \left<\alpha\cdot \beta_{t_0},U\right>\cap p_H^{-1}(\beta(t_0)).\end{equation} Put $s_n=t_0$ and $U_n=U$.

   Now, $\beta([s_{i-1},s_i])\subseteq U_i$ for $1\leq i \leq n$. For each $1\leq i \leq n-1$, choose a path-connected open subset $V_i$ of $X$ such that $\beta(s_i)\in V_i \subseteq U_i\cap U_{i+1}$.

As in the Tube Lemma \cite[Lemma~26.8]{Munkres}, for every $1\leq i \leq n$, there is an open subset $N_i\subseteq Y$ with $y_0\in N_i$ and an interval $I_i$, open in $[0,1]$, with $[s_{i-1},s_i]\subseteq I_i$ such that $F(N_i\times I_i)\subseteq U_i$. For each $1\leq i \leq n-1$, choose an open subset $M_i\subseteq Y$ with $y_0\in M_i$ and an interval $J_i$, open in $[0,1]$, with $s_i\in J_i$ such that $F(M_i\times J_i)\subseteq V_i$. Since $\widetilde{f}(y_0)=\left<\alpha\right>$ and $\alpha(1)\in U_1$, and since $\widetilde{f}$ is continuous, there is an open subset $M_0\subseteq Y$ with $y_0\in M_0$ and $\widetilde{f}(M_0)\subseteq \left<\alpha,U_1\right>$. Put $M=\Big(\bigcap_{i=1}^n N_i\Big)\cap \Big(\bigcap_{i=0}^{n-1} M_i\Big)$ and choose $I$ to be an interval, open in $[0,1]$, with $s_n\in I\subseteq (s_{n-1},1]\cap I_n$.

Let $(y,t)\in M\times I$. For notational convenience, put $\gamma_t=\beta_{y,t}$, put $\lambda_i=\beta_{y,s_{i-1},s_i}$ for $1\leq i \leq n-1$, and put $\lambda_n=\beta_{y,s_{n-1},t}$. Then $\gamma_t=\lambda_1\cdot \lambda_2 \cdot \cdots \cdot \lambda_n$ with $\lambda_i([0,1])\subseteq U_i$ for $1\leq i \leq n$ and $\lambda_i(1)\in V_i$ for $1\leq i \leq n-1$. Similarly, put $\delta_i=\beta_{s_{i-1},s_i}$ for $1\leq i \leq n$, so that $\beta_{t_0}=\delta_1\cdot \delta_2\cdot \cdots \cdot \delta_n$ with $\delta_i([0,1])\subseteq U_i$ for $1\leq i \leq n$ and $\delta_i(1)\in V_i$ for $1\leq i \leq n-1$. For each $1\leq i \leq n-1$, choose a path $\eta_i:[0,1]\rightarrow V_i\subseteq U_i\cap U_{i+1}$ from $\eta_i(0)=\delta_i(1)$ to $\eta_i(1)=\lambda_i(1)$.

Since $\widetilde{f}(y)\in \left<\alpha,U_1\right>$, there is a path $\eta_0:([0,1],0)\rightarrow (U_1,\alpha(1))$ such that $\left<\alpha_y\right>=\widetilde{f}(y)=\left<\alpha\cdot \eta_0\right>$. Hence, $\widetilde{F}(y,t)=\left<\alpha_y\cdot \beta_{y,t}\right>=\left<\alpha\cdot \eta_0\cdot \lambda_1 \cdot \lambda _2 \cdot \cdots \cdot \lambda_n\right>$. We wish to show that $\widetilde{F}(y,t)\in \left<\alpha\cdot \beta_{t_0},U\right>=\left<\alpha\cdot \delta_1\cdot \delta_2\cdot \cdots \cdot \delta_n,U\right>$. That is, we wish to find a path $\epsilon:([0,1],0)\rightarrow (U,\beta(t_0))$ and an element $h\in H$ such that \begin{equation}\label{goal} [\alpha\cdot \eta_0 \cdot \lambda_1 \cdot \lambda _2 \cdot \cdots \cdot \lambda_n]=h[\alpha\cdot \delta_1\cdot \delta_2\cdot \cdots \cdot \delta_n\cdot \epsilon].\end{equation}

Note that $\left<\alpha\cdot \delta_1 \cdot  \delta_1^-\cdot \eta_0 \cdot \lambda_1\cdot \eta_1^-\right>\in \left<\alpha\cdot \beta_{s_1},U_1\right>\cap p_H^{-1}(\beta(s_1))$. By (\ref{starstar}), we have $\left<\alpha\cdot \delta_1\cdot \delta_1^-\cdot \eta_0\cdot \lambda_1\cdot \eta_1^-\cdot \delta_2\cdot \delta_3 \cdot \cdots \cdot \delta_n\right>=\phi_{\delta_2\cdot \delta_3\cdot \cdots \cdot \delta_n}(\left<\alpha\cdot \delta_1 \cdot \delta_1^-\cdot \eta_0 \cdot \lambda_1\cdot \eta_1^-\right>)=\left<\alpha\cdot \beta_{t_0}\cdot \epsilon_1\right>$ for some loop $\epsilon_1:[0,1]\rightarrow U$ with $\epsilon_1(0)=\epsilon_1(1)=\beta(t_0)$. Hence, we have $[\alpha\cdot \eta_0\cdot \lambda_1\cdot \eta_1^- \cdot \delta_2\cdot \delta_3\cdot \cdots \cdot \delta_n]=h_1[\alpha\cdot \beta_{t_0}\cdot \epsilon_1]$ for some $h_1\in H$. Put $a_1=[\alpha\cdot \beta_{t_0}\cdot \epsilon_1 \cdot \beta_{t_0}^-\cdot \alpha^-]\in\pi(\alpha\cdot \beta_{t_0},U)$. Then
$[\alpha\cdot \eta_0\cdot \lambda_1]=h_1a_1[\alpha\cdot \delta_1\cdot\eta_1]$ so that
\[[\alpha\cdot \eta_0 \cdot \lambda_1 \cdot \lambda _2 \cdot \cdots \cdot \lambda_n]=h_1a_1[\alpha\cdot \delta_1\cdot \eta_1 \cdot \lambda_2 \cdot \lambda_3 \cdot \cdots \cdot \lambda_n].\]
Similarly, $\left<\alpha\cdot \delta_1\cdot \delta_2 \cdot \delta_2^- \cdot \eta_1\cdot \lambda_2 \cdot \eta_2^-\right>\in \left<\alpha\cdot \beta_{s_2},U_2\right>\cap p_H^{-1}(\beta(s_2))$. By (\ref{starstar}), \linebreak $\left<\alpha\cdot \delta_1\cdot \delta_2 \cdot \delta_2^- \cdot \eta_1\cdot \lambda_2 \cdot \eta_2^-\cdot \delta_3 \cdot  \cdots \cdot \delta_n\right>=\phi_{\delta_3\cdot \cdots \cdot \delta_n}(\left<\alpha\cdot \delta_1\cdot \delta_2 \cdot \delta_2^- \cdot \eta_1\cdot \lambda_2 \cdot \eta_2^-\right>)=\left<\alpha\cdot \beta_{t_0}\cdot \epsilon_2\right>$ for some loop $\epsilon_2:[0,1]\rightarrow U$ with $\epsilon_2(0)=\epsilon_2(1)=\beta(t_0)$. Hence, $[\alpha\cdot \delta_1\cdot \eta_1\cdot \lambda_2 \cdot \eta_2^-\cdot \delta_3 \cdot \cdots \cdot \delta_n]=h_2[\alpha\cdot \beta_{t_0}\cdot \epsilon_2]$ for some $h_2\in H$. Put
$a_2=[\alpha\cdot \beta_{t_0}\cdot \epsilon_2 \cdot \beta_{t_0}^-\cdot \alpha^-]\in\pi(\alpha\cdot \beta_{t_0},U)$. Then
$[\alpha\cdot \delta_1\cdot \eta_1\cdot \lambda_2]=h_2a_2[\alpha\cdot \delta_1\cdot \delta_2 \cdot \eta_2]$ so that
\[[\alpha\cdot \eta_0\cdot \lambda_1\cdot \lambda_2\cdot \cdots \cdot \lambda_n]=h_1a_1h_2a_2[\alpha\cdot \delta_1\cdot \delta_2 \cdot \eta_2\cdot \lambda_3 \cdot \lambda_4\cdot \cdots \cdot \lambda_n].\]
Inductively, we find $h_1, h_2, \dots , h_{n-1}\in H$ and loops $\epsilon_1, \epsilon_2,\dots, \epsilon_{n-1}:[0,1]\rightarrow U$ with $\epsilon_i(0)=\epsilon_i(1)=\beta(t_0)$ so that \[[\alpha\cdot \eta_0\cdot \lambda_1\cdot \lambda_2\cdot \cdots \cdot \lambda_n]=h_1a_1h_2a_2\cdots h_{n-1}a_{n-1} [\alpha\cdot \delta_1\cdot \delta_2 \cdot \cdots \cdot \delta_{n-1}\cdot \eta_{n-1}  \cdot \lambda_n]\] with
$a_i=[\alpha\cdot \beta_{t_0}\cdot \epsilon_i \cdot \beta_{t_0}^-\cdot \alpha^-]\in\pi(\alpha\cdot \beta_{t_0},U)$.

Since $\pi(\alpha \cdot \beta_{t_0},U) H= H\pi(\alpha\cdot \beta_{t_0} ,U)$, there is an $h\in H$ and a loop $\epsilon_0:[0,1] \rightarrow U$ with $\epsilon_0(0)=\epsilon_0(1)=\beta(t_0)$ such that $h_1a_1h_2a_2\cdots h_{n-1}a_{n-1}=ha$, where $a=[\alpha\cdot \beta_{t_0} \cdot \epsilon_0 \cdot \beta_{t_0}^-\cdot \alpha^-]\in \pi(\alpha \cdot \beta_{t_0},U)$. Put $\epsilon=\epsilon_0\cdot \delta_n^-\cdot \eta_{n-1}\cdot \lambda_n$. Then the path $\epsilon:([0,1],0)\rightarrow (U,\beta(t_0))$  and the element $h\in H$ satisfy (\ref{goal}).
\end{proof}

\begin{lemma}\label{slide} Let $H\leq \pi_1(X,x_0)$ and  $\left<\alpha\cdot \beta \cdot \gamma\right>\in \widetilde{X}_H$.  Let $U,V\subseteq X$ be open subsets with $\beta([0,1])\subseteq V$ and $\gamma(1)\in U$.  If $\phi_{\beta\cdot \gamma}(\left<\alpha,V\right>\cap p_H^{-1}(\beta(0)))\subseteq \left<\alpha\cdot \beta\cdot \gamma,U\right>$, then $\phi_{\gamma}(\left<\alpha\cdot \beta,V\right>\cap p_H^{-1}(\gamma(0)))\subseteq \left<\alpha\cdot \beta\cdot \gamma,U\right>$.
\end{lemma}

\begin{proof}
Let $\left<\eta\right>\in \left<\alpha\cdot \beta,V\right>\cap p_H^{-1}(\gamma(0))$. Then $\left<\eta\right>=\left<\alpha\cdot \beta\cdot \delta\right>$ for some loop $\delta$ in $V$. So, $\phi_\gamma(\left<\eta\right>)=\phi_{\beta\cdot \gamma}(\left<\alpha\cdot \beta\cdot \delta \cdot \beta^-\right>)\in \left<\alpha\cdot \beta\cdot \gamma,U\right>$.
\end{proof}

 \section{The homotopically path Hausdorff property}\label{HpH}

\begin{definition}\label{HpHDef}
We call $X$ {\em homotopically path Hausdorff} relative to a subgroup $H\leq \pi_1(X,x_0)$ if for every two paths $\alpha,\beta:([0,1],0)\rightarrow (X,x_0)$ with $\alpha(1)=\beta(1)$ and $[\alpha\cdot \beta^-]\not \in H$, there is a partition $0=t_0<t_1<\cdots < t_n=1$ of $[0,1]$ and open subsets $U_1, U_2, \dots, U_n$ of $X$ with $\alpha([t_{i-1},t_i])\subseteq U_i$ for all $1\leq i\leq n$ and such that if $\gamma:[0,1]\rightarrow X$ is any path with $\gamma([t_{i-1},t_i])\subseteq U_i$ for all $1\leq i\leq n$ and with $\gamma(t_i)=\alpha(t_i)$ for all $0\leq i \leq n$, then $[\gamma\cdot \beta^-]\not\in H$.
\end{definition}

\begin{Remark} If $X$ is homotopically path Hausdorff relative to $H\leq \pi_1(X,x_0)$, then $p_H:\widetilde{X}_H\rightarrow X$ has unique path lifting \cite{BFa,FRVZ}. The converse does not hold in general. The exact difference between these two notions is discussed in \cite{BFi}.
\end{Remark}

\begin{definition}\label{LQNC}
We call a subgroup $H\leq \pi_1(X,x_0)$ {\em locally quasinormal at the constant path} if for every open neighborhood $U$ of $x_0$ in $X$, there is an open neighborhood $V$ of $x_0$ in $X$ with $x_0\in V \subseteq U$ such that $H \pi(c,V)=\pi(c,V)H$, where $c:[0,1]\rightarrow \{x_0\}$ denotes the constant path at $x_0$.
\end{definition}

\begin{Remark}
If $H\leq \pi_1(X,x_0)$ is locally quasinormal, then $H$ is locally quasinormal at the constant path.
\end{Remark}

In this section, we prove the following implication.

\begin{theorem}\label{thm3} Let $H\leq \pi_1(X,x_0)$ be locally quasinormal at the constant path. Suppose that $p_H^{-1}(x_0)$ is $T_1$ and that for every $x\in X$ and every path $\beta:[0,1]\rightarrow X$ from $\beta(0)=x$ to $\beta(1)=x_0$, the monodromy $\phi_\beta:p_H^{-1}(x)\rightarrow p_H^{-1}(x_0)$ is continuous. Then $X$ is homotopically path Hausdorff relative to $H$.
\end{theorem}

\begin{proof}
Let $\alpha, \beta:([0,1],0)\rightarrow (X,x_0)$ be two paths with $\alpha(1)=\beta(1)$ and $[\alpha\cdot \beta^-]\not\in H$. Let $c:[0,1]\rightarrow \{x_0\}$ denote the constant path at $x_0$. Then $\left<c\right>$ and $\left<\beta\cdot \alpha^-\right>$ are distinct elements of the fiber $p_H^{-1}(x_0)$. Hence, there is an open neighborhood $U$ of $x_0$ in $X$ such that $\left<\beta\cdot \alpha^-\right>\not\in \left<c,U\right>\cap p_H^{-1}(x_0)$. We may assume that $H\pi(c,U)=\pi(c,U)H$. Put $\alpha_t(s)=\alpha(ts)$.

For every $t\in [0,1]$, the monodromy $\phi_{\alpha_t^-}:p_H^{-1}(\alpha(t))\rightarrow p_H^{-1}(x_0)$ is continuous. Hence, for each $t\in [0,1]$, there is an open subset $V_t\subseteq X$ with $\alpha(t)\in V_t$ such that \[\phi_{\alpha_t^-}(\left<\alpha_t,V_t\right>\cap p_H^{-1}(\alpha(t)))\subseteq \left<c,U\right>\cap p_H^{-1}(x_0).\] As in the proof of Theorem~\ref{thm2}, we find a
 subdivision $0=t_0<t_1<\cdots <t_n=1$ and open subsets $V_0, V_1, \dots, V_{n-1}\subseteq X$ such that, for every $0\leq i \leq n-1$, we have $\alpha([t_i,t_{i+1}])\subseteq V_i$ and \[\phi_{\alpha_{t_i}^-}(\left<\alpha_{t_i},V_i\right>\cap p_H^{-1}(\alpha(t_i)))\subseteq \left<c,U\right>\cap p_H^{-1}(x_0),\] where $V_0=U$.

Now, let $\gamma:[0,1]\rightarrow X$ be any path such that $\gamma([t_i,t_{i+1}])\subseteq V_i$ for $0\leq i \leq n-1$ and $\gamma(t_i)=\alpha(t_i)$ for $0\leq i \leq n$. We wish to show that $[\gamma\cdot \beta^-]\not\in H$.

For $1\leq i \leq n$, define paths $\delta_i, \lambda_i:[0,1]\rightarrow X$ by $\delta_i(s)=\alpha(t_{i-i}+s(t_i-t_{i-1}))$ and $\lambda_i(s)=\gamma(t_{i-i}+s(t_i-t_{i-1}))$. Then $\alpha_{t_i}=\delta_1\cdot \delta_2\cdot \cdots \cdot \delta_i$ and $\gamma=\lambda_1\cdot \lambda_2\cdot \cdots \cdot \lambda_n$.

Since $\left<\delta_1\cdot \lambda_2\cdot \delta_2^-\right>\in\left<\delta_1,V_1\right>\cap p_H^{-1}(\alpha(t_1))$, we have \[\left<\delta_1\cdot \lambda_2\cdot \delta_2^-\cdot \delta_1^-\right>=\phi_{\delta_1^-}(\left<\delta_1\cdot \lambda_2\cdot \delta_2^-\right>)\in \left<c,U\right>\cap p_H^{-1}(x_0).\] Then $\left<\delta_1\cdot \lambda_2\cdot \delta_2^-\cdot \delta_1^-\right>=\left<\epsilon_1\right>$ for some loop $\epsilon_1:[0,1]\rightarrow U$ with $\epsilon_1(0)=\epsilon_1(1)=x_0$. Hence, $[\delta_1\cdot \lambda_2\cdot \delta_2^-\cdot \delta_1^-]=h_1[\epsilon_1]$ for some $h_1\in H$.

Since $\left<\delta_1\cdot \delta_2\cdot \lambda_3\cdot \delta_3^-\right>\in \left<\delta_1\cdot \delta_2, V_2\right>\cap p_H^{-1}(\alpha(t_2))$, we have \[\left<\delta_1\cdot \delta_2\cdot \lambda_3\cdot \delta_3^-\cdot \delta_2^-\cdot \delta_1^- \right>=\phi_{\delta_2^-\cdot \delta_1^-}(\left<\delta_1\cdot \delta_2\cdot \lambda_3\cdot \delta_3^-\right>)\in \left<c,U\right>\cap p_H^{-1}(x_0).\] Consequently, $[\delta_1\cdot \delta_2\cdot \lambda_3\cdot \delta_3^-\cdot \delta_2^-\cdot \delta_1^-]=h_2[\epsilon_2]$ for some $h_2\in H$ and some loop $\epsilon_2:[0,1]\rightarrow U$ with $\epsilon_2(0)=\epsilon_2(1)=x_0$.
Similarly, for every $1\leq i\leq n-1$, there is an $h_i\in H$ and a loop $\epsilon_i:[0,1]\rightarrow U$ with $\epsilon_i(0)=\epsilon_i(1)=x_0$ such that \[[\delta_1\cdot \cdots \cdot \delta_i\cdot \lambda_{i+1}\cdot \delta_{i+1}^-\cdot \delta_i^-\cdot \cdots \cdot \delta_1^-]=h_i[\epsilon_i].\]
We also put $\epsilon_0=\lambda_1\cdot \delta_1^-$. Then $[\epsilon_i]\in \pi(c,U)$ for all $0\leq i \leq n-1$ and \[[\gamma\cdot \alpha^-]=[\epsilon_0]h_1[\epsilon_1]h_2 \cdots [\epsilon_{n-1}]h_{n-1}[\epsilon_{n-1}].\] Since  $\pi(c,U)H =H\pi(c,U)$, there is an $h\in H$ and a loop $\epsilon:[0,1]\rightarrow U$ with $\epsilon(0)=\epsilon(1)=x_0$ such that $[\gamma\cdot \alpha^-]=h[\epsilon]$. Hence, $\left<\gamma\cdot \alpha^-\right>=\left<\epsilon\right>\in \left<c,U\right>\cap p_H^{-1}(x_0)$, so that $\left<\gamma\cdot \alpha^-\right>\not=\left<\beta\cdot \alpha^-\right>$. Consequently, $[\gamma\cdot \beta^-]\not\in H$.

\end{proof}

\section{Continuous monodromy versus small loop transfer}\label{SLT}

We now examine how our results relate to Theorem~2.5 of \cite{PMTAR}, which is stated in terms of small loop transfer, rather than continuous monodromy.
The authors of \cite{PMTAR} introduced the following relative version of small loop transfer (SLT) that had previously been defined in \cite{BDLM} for $H=\{1\}$.

\begin{definition} \label{SLT-Def} Let $H\leq \pi_1(X,x_0)$. We call $X$ an {\em $H$-SLT space at $x_0$} if for every path $\alpha:([0,1],0)\rightarrow (X,x_0)$ and for every open neighborhood $U$ of $x_0$ in $X$, there is an open neighborhood $V$ of $\alpha(1)$ in $X$ such that for every loop $\beta:([0,1],\{0,1\})\rightarrow (V,\alpha(1))$, there is a loop $\gamma:([0,1],\{0,1\})\rightarrow (U,x_0)$ with $[\alpha\cdot \beta\cdot \alpha^-\cdot \gamma]\in H$. We call $X$ an {\em $H$-SLT space} if for every $x\in X$ and for every path $\delta:[0,1]\rightarrow X$ from $\delta(0)=x_0$ to $\delta(1)=x$, $X$ is a $[\delta^-]H[\delta]$-SLT space at $x$.
\end{definition}

\begin{Remark} For locally path-connected $X$ and normal $H \trianglelefteq \pi_1(X,x_0)$, we have that $X$ is an $H$-SLT space if and only if for every $\delta:([0,1],0)\rightarrow (X,x_0)$, the topology of $\widetilde{X}_{[\delta^-]H[\delta]}$ agrees with the quotient topology of the path space in the compact-open-topology. (See \cite{BDLM} and \cite{PMTAR}.)
\end{Remark}

A proof of the following relationship can be found in \cite[Proposition 2.13]{PMTAR}:

\begin{lemma} \label{SLT-Mon} Let $H\leq\pi_1(X,x_0)$. Then the following are equivalent:
\begin{itemize}
\item[(i)] For every $x,y\in X$ and every path $\beta:[0,1]\rightarrow X$ from $\beta(0)=x$ to $\beta(1)=y$, the monodromy $\phi_\beta:p_H^{-1}(x)\rightarrow p_H^{-1}(y)$ is continuous.
\item[(ii)] $X$ is an $H$-SLT space.
\end{itemize}
\end{lemma}

In fact, the same proof yields:

\begin{lemma} \label{SLT-Conj} Let $H\leq \pi_1(X,x_0)$. Then the following are equivalent:
\begin{itemize}
\item[(i)] For every $x\in X$ and every path $\beta:[0,1]\rightarrow X$ from $\beta(0)=x$ to $\beta(1)=x_0$, the monodromy $\phi_\beta:p_H^{-1}(x)\rightarrow p_H^{-1}(x_0)$ is continuous.
\item[(ii)] For every $\delta:([0,1],\{0,1\})\rightarrow (X,x_0)$, $X$ is a $[\delta^-]H[\delta]$-SLT space at $x_0$.
\end{itemize}
\end{lemma}

\begin{lemma}\label{SLT-N}
 For a normal subgroup $H \trianglelefteq \pi_1(X,x_0)$, the following are equivalent:
\begin{itemize}
\item[(i)] For every $x\in X$ and every path $\beta:[0,1]\rightarrow X$ from $\beta(0)=x$ to $\beta(1)=x_0$, the monodromy $\phi_\beta:p_H^{-1}(x)\rightarrow p_H^{-1}(x_0)$ is continuous.
\item[(ii)] $X$ is an $H$-SLT space at $x_0$.
\end{itemize}
\end{lemma}

Therefore, we may state one of the main results of \cite{PMTAR} as follows:

\begin{theorem}[Theorem~2.5 of \cite{PMTAR}] \label{thm}
Let $H \trianglelefteq \pi_1(X,x_0)$ be a normal subgroup. Suppose that $p^{-1}(x)$ is $T_1$ for every $x\in X$ and that for every $x\in X$ and every path $\beta:[0,1]\rightarrow X$ from $\beta(0)=x$ to $\beta(1)=x_0$, the monodromy $\phi_\beta:p_H^{-1}(x)\rightarrow p_H^{-1}(x_0)$ is continuous. Then $X$ is homotopically path Hausdorff relative to $H$.
\end{theorem}

We note that both Theorem~\ref{thm1} and Theorem~\ref{thm2} are extensions of Theorem~\ref{thm} and that Theorem~\ref{thm3} is a generalization of Theorem~\ref{thm}.
\vspace{10pt}

\noindent {\bf Acknowledgements.} The first author was partially supported by a grant from the Simons Foundation (No.\ 245042).

\end{document}